\documentclass[11pt,a4paper]{amsart}
\usepackage{a4wide}
\usepackage{amsfonts,amsthm,amsmath,amssymb}
\usepackage{amscd,color,caption}
\usepackage{psfrag,graphicx, pgfplots,subfigure,float,subfloat}
\usepackage{import,algorithmic,algorithm2e}
\usepackage{comment}
\usepackage{enumitem}
\usepackage{theoremref}
\usepackage{mathtools}

\usepackage{xcolor}
\usepackage{subfig}
\usepackage{subfigure}
\usepackage{verbatim}

\tikzset{ font={\fontsize{9pt}{12}\selectfont}}
\newcommand{\overbar}[1]{\mkern 1.5mu\overline{\mkern-1.5mu#1\mkern-1.5mu}\mkern 1.5mu}

\newtheorem*{thmA}{Theorem A}
\newtheorem*{thmB}{Theorem B}

\theoremstyle{plain}

\newtheorem{theorem}{Theorem}[section]
\newtheorem{lemma}[theorem]{Lemma}
\newtheorem{corollary}[theorem]{Corollary}
\newtheorem{prop}[theorem]{Proposition}

\newtheorem{remark}{Remark}

\newcommand{\rs}{\hat{\mathbb{C}}}

\newcommand{\immo}{A_{n,\alpha}^{*}(1)}
\newcommand{\mmodoi}{A_{B_a}(\infty)}
\newcommand{\immodoi}{A_{B_a}^{*}(\infty)}

\newcommand{\immR}{A_{R_{n,\alpha}}^{*}(\infty)}
\newcommand{\mmo}{A_{n,\alpha}(1)}

\newcommand{\jn}{O_{n, \alpha}}
\newcommand{\jdoi}{B_{ a}}
\newcommand{\jndoi}{O_{2, \alpha}}
\newcommand{\jtrei}{O_{3, \alpha}}
\newcommand{\jntrei}{R_{n,\alpha}}
\newcommand{\jnconj}{R_{n,\alpha}}
\newcommand{\jnconju}{R_{n,\alpha}}
\newcommand{\jnconjt}{T_{n,\alpha}}
\newcommand{\jnt}{T_{n,\alpha}}
\newcommand{\jnc}{O_{n, \alpha, c}}
\newcommand{\jnmi}{O_{n, \alpha, -1}}

\newsavebox{\savepar}

\graphicspath{{./figures/}}

\begin{document}

\title{Newton-like components in the Cebyshev-Halley family of degree $n$ polynomials}

\date{\today}

\author{Dan Paraschiv}
\address{Departament de Matem\`atiques i Inform\`atica at Universitat de Barcelona and Barcelona Graduate School of Mathematics, 08007 Barcelona, Catalonia.}
\email{dan.paraschiv@ub.edu}

\thanks{The author was  supported by the Spanish government grant FPI PRE2018-086831. The author was also supported by PID2020-118281GB-C32.}\

\begin{abstract}
We study the Cebyshev-Halley methods applied to the family of polynomials $f_{n,c}(z)=z^n+c$, for $n\ge 2$ and $c\in \mathbb{C}^{*}$. We prove the existence of parameters such that the immediate basins of attraction corresponding to the roots of unity are infinitely connected. We also prove that, for $n \ge 2$, the corresponding dynamical plane contains a connected component of the Julia set, which is a quasiconformal deformation of the Julia set of the map obtained by applying Newton's method to $f_{n,-1}$.
\vspace{0.5cm}


\end{abstract}

\maketitle

\section{Introduction}

Numerical methods have been extensively used to give accurate approximations of the solutions of systems of non-linear equations. Those equations or systems of equations correspond to a wide source of scientific models from biology to engineering and from economics to social sciences, and so their solutions are the cornerstone of applied mathematics. One of the most sets families of numerical methods are the so called 
root-finding algorithms; that is, iterative methods which asymptotically converge to the zeros (or some of the zeros) of the non linear equation, say $g(z)=0$. Although $g$ could in general describe an arbitrary high dimensional problem, in this paper we focus on the one dimensional case, i.e. $g:\mathbb C \to \mathbb C$.

The universal and most studied root-finding algorithm is known as {\it Newton’s method}. 
If $g$ is holomorphic, the Newton’s method applied to $g$ is the iterative root finding algorithm defined as follows 
$$
z_{n+1} = z_n - \frac{g(z_n)}{g^{\prime}(z_n)}, \ \ z_0\in \mathbb C.
$$
It is well known that if $z_0\in \mathbb C$ is chosen close enough to one of the solutions of the equation $g(z)=0$, say $\alpha$, then the sequence $\{z_n=g^n(z_0)\}_{n\geq 0}$ has the limit $\alpha$ when $n$ tends to $\infty$. Moreover the speed of (local) convergence is generically quadratic (see, for instance,  \cite{AmatBusqRev}). It was Cayley (see \cite{Cayley}) the first to consider Newton's method as a (holomorphic) dynamical system, that is studying the convergence of these sequences for all possible seeds $x_0\in \mathbb C$ at once, under the assumption that $g$ was a degree 2 or 3 polynomial. This was known as {\it Cayley's problem}.

Many authors have studied alternative iterative methods having, for instance, a better local speed of convergence. Two of the best known root-finding algorithms of order of convergence $3$ are {\it Chebyshev's method} and {\it Halley's method} (see \cite{AmatBusqRev}). They are included in the {\it Chebyshev-Halley family} of root-finding algorithms, which was introduced in \cite{CTV} (see also \cite{AmatBusqAdv}), and is defined as follows. Let $g$ be a holomorphic map. Then
\begin{equation}z_{n+1}=z_n-\left( 1+\frac{1}{2}\,\frac{L_g(z_n)}{1-\alpha L_g(z_n)} \right)\frac{g(z_n)}{g'(z_n)},\end{equation}
where $\alpha \in [ 0, \, 1]$ and $L_g(z)=\frac{g(z)g"(z)}{(g'(z))^2}$. Notice that in a real setting, it suffices for $g$ to be a doubly differentiable function such that $g''(x)$ is continuous. 

For $\alpha=0$, we have Chebyshev's method and for $\alpha=\frac{1}{2}$ Halley's method. As $\alpha$ tends to $\infty$, the Cebyshev-Halley algorithms tend to Newton's method. The main goal of the paper is to show that the unbounded connected component of the Julia set of the Chebyshev-Halley maps applied to $z^{n}-c$ (for large enough $\alpha$) is homeomorphic to the Julia set of the map obtained by applying Newton's method to $z^n -1$.

Next, we give a brief introduction of complex dynamics for holomorphic maps defined over the Riemann sphere, that is, rational maps. For a more detailed description of the topic, see, for instance, \cite{Bear} and \cite{Mi1}. Let $Q:\rs \to \rs$ be a rational map of degree $d \ge 2$. A point $z_0 \in \rs$ is called a \textit{fixed point} if $Q(z_0)=z_0$. The \textit{multiplier} $\lambda$ of the fixed point $z_0$ is $\lambda\coloneqq Q' (z_0)$. If $|\lambda|<1$, then $z_0$ is an \textit{attracting fixed point}. If $\lambda=0$, we say that $z_0 $ is a {\it superattracting fixed point}. The \textit{basin of attraction} of an attracting fixed point $z_0$ is the set $A_Q (z_0)= \{ z\in \rs | \lim_{n\to \infty} Q^{n}(z)=z_0\}$. We denote by $A_Q^{*}(z_0)$ the connected component of the basin of attraction which contains $z_0$, also known as the \textit{immediate basin of attraction} of $z_0$. Any immediate basin of attraction contains at least one critical point, that is, a point $c\in \rs$ such that $Q'(c)=0$.

The Fatou set $F(Q)$ is defined as the set of points of normality. A point $z\in \rs$ is said to be normal if the family $\{Q^n\}_{n\ge 1}$ is normal for some neighborhood $U$ of $z$. A connected component of the Fatou set is called a Fatou component. The complement of the Fatou set is called the Julia set, denoted by $J(Q)$. Both Fatou and Julia sets are dinamically invariant. The behaviour of $Q$ on its corresponding Julia set $J(Q)$ is chaotic. Moreover, the Julia set $J(Q)$ is non-empty. 

Root-finding algorithms are a natural topic for complex dynamics. In particular, maps obtained by applying Newton's method to polynomials are a much studied topic (see \cite{Shi0}, \cite{HSS}, \cite{Ta}, \cite{BFJK}). It is proven in \cite{Shi0} that the Julia set of such maps is connected, so all Fatou components are simply connected.

Previously, Campos, Canela, and Vindel have studied the Cebyshev-Halley family applied to $f_{n,c}(z)=z^n+c, \,c \in \mathbb{C}^{*}$ (see \cite{CCV}, \cite{CCV1}). The maps obtained by applying the Cebyshev-Halley family to $f_{n,c}$ are all conjugated to the map obtained by applying the Cebyshev-Halley family to $f_n (z)\coloneqq f_{n, -1}(z)=z^n-1$ (see \thref{jordica}). By applying the Cebyshev-Halley method to $f_n(z)=z^n-1$ we obtain the map:
\begin{equation}
\label{exprjn}
\jn(z)=\frac{(1-2\alpha)(n-1)+(2-4\alpha-4n+6\alpha n-2\alpha n^2)z^n+(n-1)(1-2\alpha -2n +2\alpha n )z^{2n}}{2nz^{n-1}(\alpha(1-n)+(-\alpha -n+\alpha n) z^n)}.\end{equation}
The map $\jn$ has degree $2n$ and it has $4n-2$ critical points, counting multiplicity. The point $z=0$ is a critical point of multiplicity $n-2$, which is mapped to the fixed point $z=\infty$. The $n$-th roots of the unity are superattracting fixed points of local degree $3$, and therefore, they have multiplicity $2$ as critical points. This leaves $n$ free critical points. They are given by
\begin{equation}
\label{exprcpjn}
c_{n, \alpha, \xi}=\xi \left( \frac{\alpha (n-1)^2(2\alpha-1)}{n(2n-1)-\alpha(4n-1)(n-1)+2\alpha^2(n-1)^2}\right)^{\frac{1}{n}},
 \end{equation}
where $\xi^n=1$. This family is symmetric with respect to rotation by the $n$-th root of the unity (see \thref{rotationwithn}). This symmetry ties the orbits of the $n$ free critical points, so the family $\jn$ has only one degree of freedom (see Figure \ref{fig:Param}).
\begin{figure}
\includegraphics[width=\textwidth]{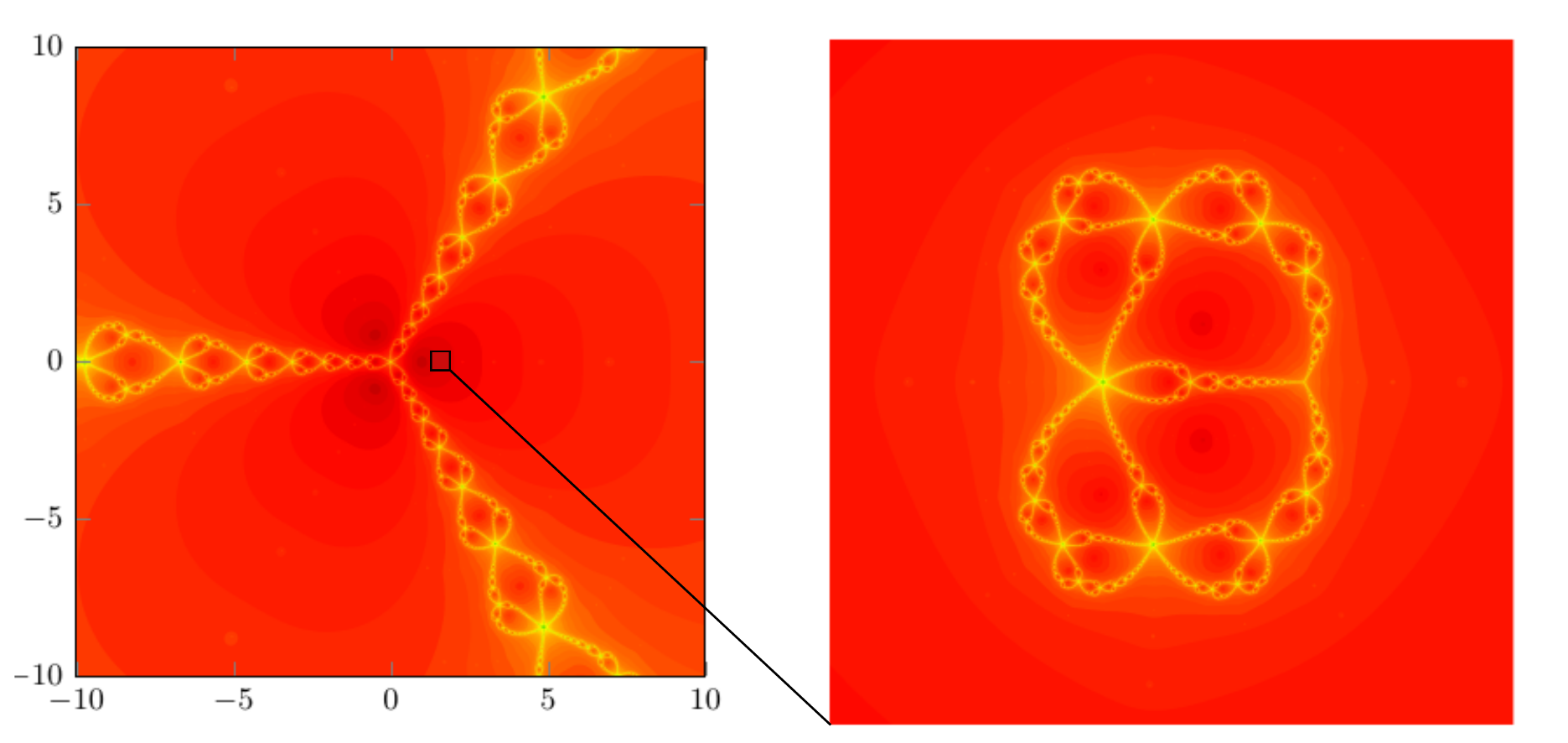}
\caption{Left figure illustrates the dynamical plane of $\jn$ for $n=3$ and $\alpha=10$. In the right figure (which shows $z\in \mathbb{C}$ such that $\textrm{Re}(z) \in [1.620; \,1.623]$ and $\textrm{Im}(z) \in[-0.0015; \, 0.0015]$ in the same dynamical plane), we can see a component of the Julia set which lies in $A^{*}(1)$.}
	\label{fig:planparam}
\end{figure}

In \cite{CCV1}, the authors studied in detail the topology of the immediate basins of attraction of the fixed points of $\jn(z)$ given by the $n-$th root of unity, that is, the zeros of $f_n(z)$. In what follows we refer to these basins as $$A_{n,\alpha}(\xi)\coloneqq A_{\jn}(\xi) \, \Big[ A_{n, \alpha}^{*}(\xi)\coloneqq A_{\jn}^{*}(\xi)\Big],$$ where $\xi^n=1$. For one particular case, the immediate basins of attraction are infinitely connected (see Figure \ref{fig:planparam}). We study the Julia set of $\jn$ for this particular case and relate it to the Julia set of the map obtained by applying Newton's method to $f_n$. We realise this using a quasiconformal surgery construction, which erases the holes in the immediate basins of attraction. The construction is a simpler case of one in \cite{McM1}. However, realizing the surgery is still needed, as we prove the uniqueness of the resulted quasiconformal map, to show that the quasirational map presents the necessary symmetries and is precisely $N_{f_n}$.

\begin{thmA}

Fix $n\ge 2$ and assume that $A_{n,\alpha}^{*}(1)$ is infinitely connected for some $\alpha \in \mathbb{C}$. Then there exists an invariant Julia component $\Pi$ (which contains $z=\infty$) which is a quasiconformal copy of the Julia set of $N_{f_n}$, where $N_{f_n}$ is the map obtained by applying Newton's method to the polynomial $f_n(z)=z^n-1$.
\end{thmA}

We finish the paper by proving that there exist parameters such that the hypothesis of Theorem A holds. We split the proof of Theorem B in two cases, $n=2$ and $n \ge 3$. For the case $n=2$, much work was previously done in \cite{CCV}, and the map is conjugate to a Blaschke product. For the case $n\ge 3 $, the map is not conjugate to a Blaschke product. We provide a map conjugate to $\jn$, for which we prove, using various properties and computational arguments, that the immediate basin of attraction of $z=\infty$ is infinitely connected. Numerical computations confirm to us the existence of such hyperbolic components (see Figure \ref{fig:Param}).

\begin{thmB}

Let $n\ge 2$. Then there exists $\alpha_0>0$ large enough such that for $\alpha>\alpha_0$, $\alpha \in \mathbb{R}$, $A_{n,\alpha}^{*}(1)$ is infinitely connected. Moreover, for $n=2$, the statement is true for any $\alpha\in \mathbb{C}$ such that $|\alpha|>\alpha_0$.
\end{thmB}

\begin{figure}
\centering
{	\begin{tikzpicture}
			\begin{axis}[width=0.48\textwidth, axis equal image, scale only axis,  enlargelimits=false, axis on top]
				\addplot graphics[xmin=-1,xmax=4,ymin=-2.5,ymax=2.5] {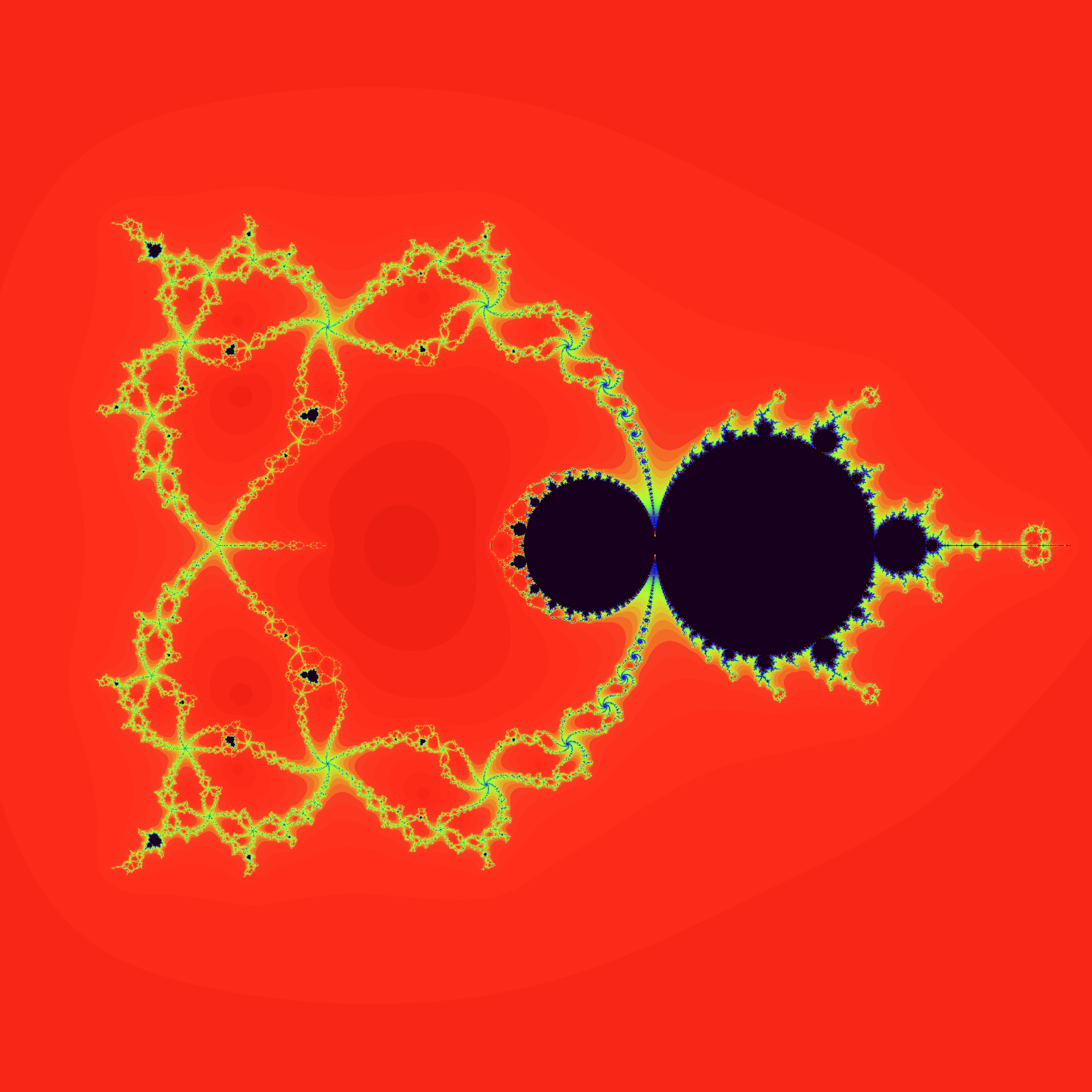};
			\end{axis}
	\end{tikzpicture}}
\hfill
{\begin{tikzpicture}
\begin{axis}[width=0.48\textwidth, axis equal image, scale only axis,  enlargelimits=false, axis on top]
				\addplot graphics[xmin=-1,xmax=4,ymin=-2.5,ymax=2.5] {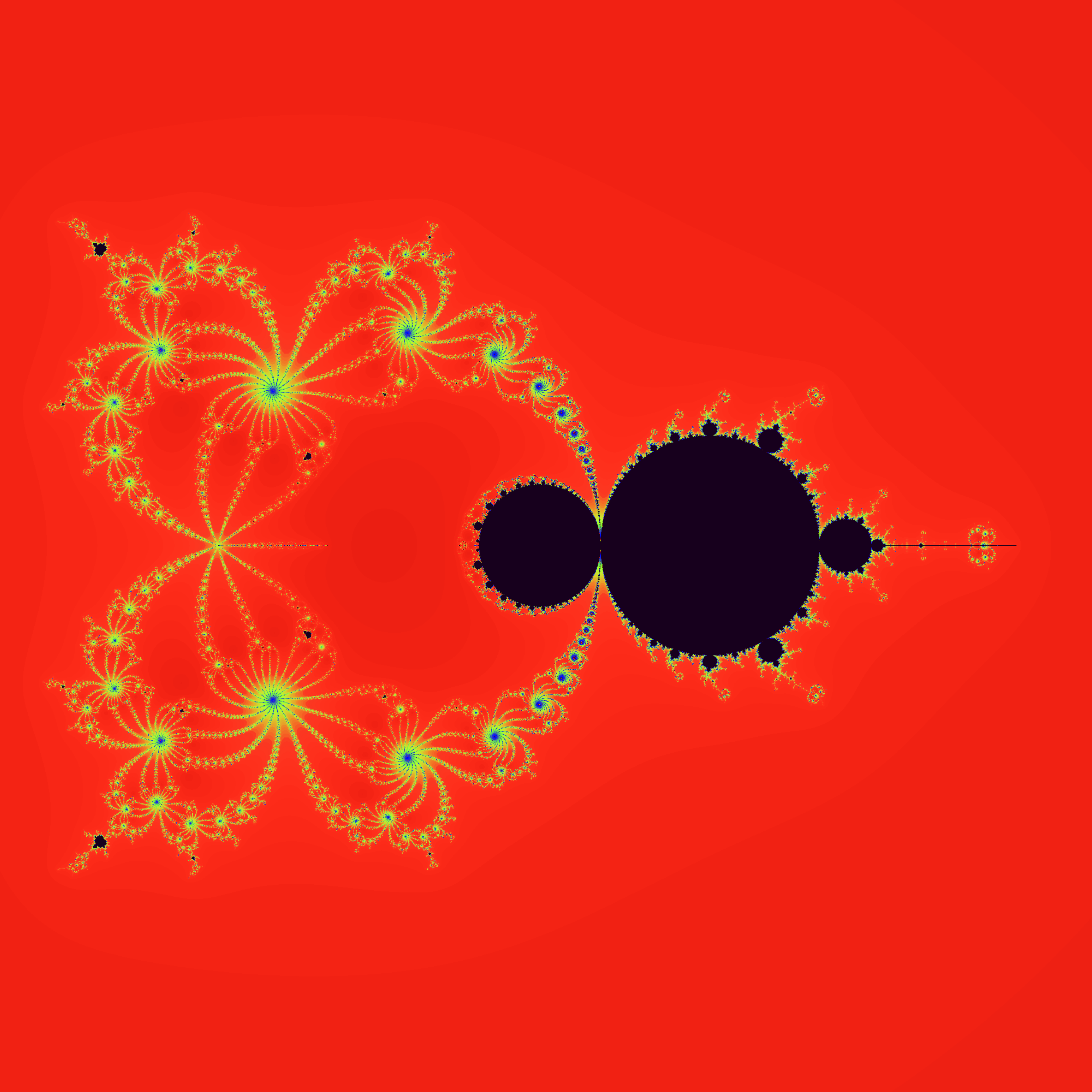};
			\end{axis}
	\end{tikzpicture}}
\caption{Left figure illustrates the parameter plane of $\jn$ for $n=3$. In the right figure we can see the parameter plane of $\jn$ for $n=5$.}
\label{fig:Param}
\end{figure}
The paper is organised as follows. In Section~\ref{sec:preliminaries} we briefly introduce the tools later used in the paper. In Section \ref{sec:ThA} we prove Theorem A. In Section \ref{sec:ThB} we prove Theorem B.

\section{Preliminaries}
\label{sec:preliminaries}

In this section we present the main tools that we use along the paper. Before, we introduce some notation. Let $U\subset \mathbb{C}$ be a multiply connected open set. We denote by $\textrm{Fill}(U)$ the minimal simply connected open set which contains $U$ but not $z=\infty$. Let $\gamma\in\mathbb C$ be a Jordan curve. We denote by ${\rm Ext}\left(\gamma\right)$ and ${\rm Int}\left(\gamma\right)$ the connected components of $\hat{\mathbb C} \setminus \gamma$ that contain $z=\infty$ and do not contain $z=\infty$, respectively. We denote the circle centered at the origin and of radius $c>0$ by $\mathbb{S}_c$. Finally, if $U \subset \mathbb C$ we denote by $\overline{U}$ its closure.

Let $\jnc$ be the map obtained by applying the Cebyshev-Halley method with parameter $\alpha$ to the map $f_{n,c}=z^n-c$. The following lemma proven in \cite{CCV} states that for any $c\in \mathbb{C}^{*}$, the map $\jnc$ is conjugated to $\jnmi=\jn$. 
\begin{lemma}
\thlabel{jordica}
Let $c \in \mathbb{C}^{*}$, and let $\eta_c(z)= \frac{z}{\sqrt[\leftroot{-2}\uproot{2}n]{-c}}$, $\eta_c:\rs \to \rs$. Then $\jnc$ and $\jnmi$ are conjugated by the map $\eta_c$, i.e. $\jnc\circ \eta_c(z)=\eta_c \circ \jnmi (z)$.
\end{lemma}

The next lemma shows that the map $\jn$ is symmetric with respect to rotation by an $n-$th root of the unity. 

\begin{lemma}[\cite{CCV}, Lemma 6.2]
\thlabel{rotationwithn}
Let $n\in \mathbb{N}$ and let $\xi$ be an $n$-th root of the unity, i.e.\ $\xi^n=1$. Then $I_{\xi}(z)\coloneqq \xi z$ conjugates $\jn$ with itself, i.e.\ $\jn \circ I_{\xi}(z) =I_{\xi} \circ \jn (z)$.
\end{lemma}

For $\alpha=\frac{1}{2}$ and $\alpha =\frac{2n-1}{2n-2}$, the family $\jn$ degenerates to maps of a lower degree (see \cite{CCV1}, Lemma 4.1). For other
values of $\alpha$, the map $\jn$ has degree $2n$, hence, it has $4n-2$ critical points. The point $z=0$ maps with degree $n-1$ to the fixed point $z=\infty$. Since the $n$ roots of $f_{n}$ are superattracting fixed points of local degree $3$, there remain precisely $n$ free critical points. The next lemma follows directly from \thref{rotationwithn}, since the orbits of the free critical points are symmetric.
\begin{lemma}[\cite{CCV1}, Lemma 3.4]
\thlabel{atmostonepoint}
Let $n \ge 2$ and $\xi \in \mathbb{C}$, such that $\xi^n=1$. For all $\alpha \in \mathbb{C}$, the basin of attraction $A_{n, \alpha, \xi}$ contains at most one critical point other than $z=1$.
\end{lemma}

The following proposition establishes a trichotomy for rational maps with the property described in \thref{atmostonepoint}. Based on the existence of the critical point and preimages of the superattracting fixed point in the immediate basin of attraction, we can establish if the immediate basin is simply connected.

\begin{prop}[\cite{CCV1}, Proposition 3.1]
\thlabel{trich}
Let $f: \rs \to \rs$ be a rational map and let $z=0$ be a superattracting fixed point of $f$. Assume that $A_f(0)$ contains at most one critical point other than $z=0$. Then, exactly one of the following statements holds.
\begin{enumerate}
\item{The set $A_f^*(0)$ contains no critical point other than $z=0$. Then $A_f^{*}(0)$ is simply connected.}
\item{The set $A_f^*(0)$ contains a critical point $c \ne 0$ and a preimage $z_0\ne 0$ of $z=0$. Then $A_f^{*}(0)$ is simply connected.}
\item{The set $A_f^*(0)$ contains a critical point $c \ne 0$ and no preimage of $z=0$ other than $z=0$ itself. Then $A_f^{*}(0)$ is multiply connected.}
\end{enumerate}
\end{prop}

\begin{corollary}[\cite{CCV1}, Corollary 3.5] 
\thlabel{cortrich}
For fixed $n \ge 2$ and $\alpha \in \mathbb{C}$, the immediate basins of attraction of the roots of $z^n-1$ under $O_{n, \alpha}$ are multiply connected if and only if $A_{n,\alpha}^{*}(1)$ contains a critical point $c \ne 1$ and no preimage of $z=1$ other than $z=1$ itself.
\end{corollary}

\begin{remark}
\thlabel{rem1}
An immediate attracting basin may only have connectivity $1$ or $\infty$ (see \cite{Bear}). Hence, if $A_{n,\alpha}^{*}(1)$ is multiply connected, then all the immediate basins of attraction corresponding to the roots of $f_n$ are infinitely connected. 
\end{remark}

The following theorem in \cite{Ta} is the critical criterion used to prove Theorem A (see also \cite{Head}).
\begin{theorem}[\cite{Ta}, Theorem 2.2]
\thlabel{tanlei}
Any rational map $F$ of degree $d$ having $d$ distinct superattracting fixed points is conjugate by a M{\"o}bius transformation to $N_P$ for a polynomial of degree $d$. Moreover, if $z=\infty$ is not superattracting for $F$ and $F$ fixes $\infty$, then $F=N_P$ for some polynomial $P$ of degree $d$.
\end{theorem}

\section{Proof of Theorem A}
\label{sec:ThA}
We start with a proposition that describes two curves in the dynamical plane of $\jn$. These curves are used in the proof of Theorem A, as part of a quasiconformal surgery construction. The proof follows closely an argument made in the proof of \cite[Proposition 3.1]{CCV1}.
\begin{prop}
\thlabel{propcurves}
Let $\jn$ such that $\immo$ is infinitely connected. Then there exist $\Gamma$ and $\Gamma^{-1}$, analytic Jordan curves in $\immo$ which surround $z=1$, such that $\jn|_{\Gamma^{-1}}: \Gamma^{-1} \to \Gamma$ is a two-to-one map with $\Gamma \subset {\rm Int}(\Gamma^{-1})$.
\end{prop}
\begin{figure}
\hfill
\subfigure[$\partial U=\gamma_2$.]{\includegraphics[width=7cm]{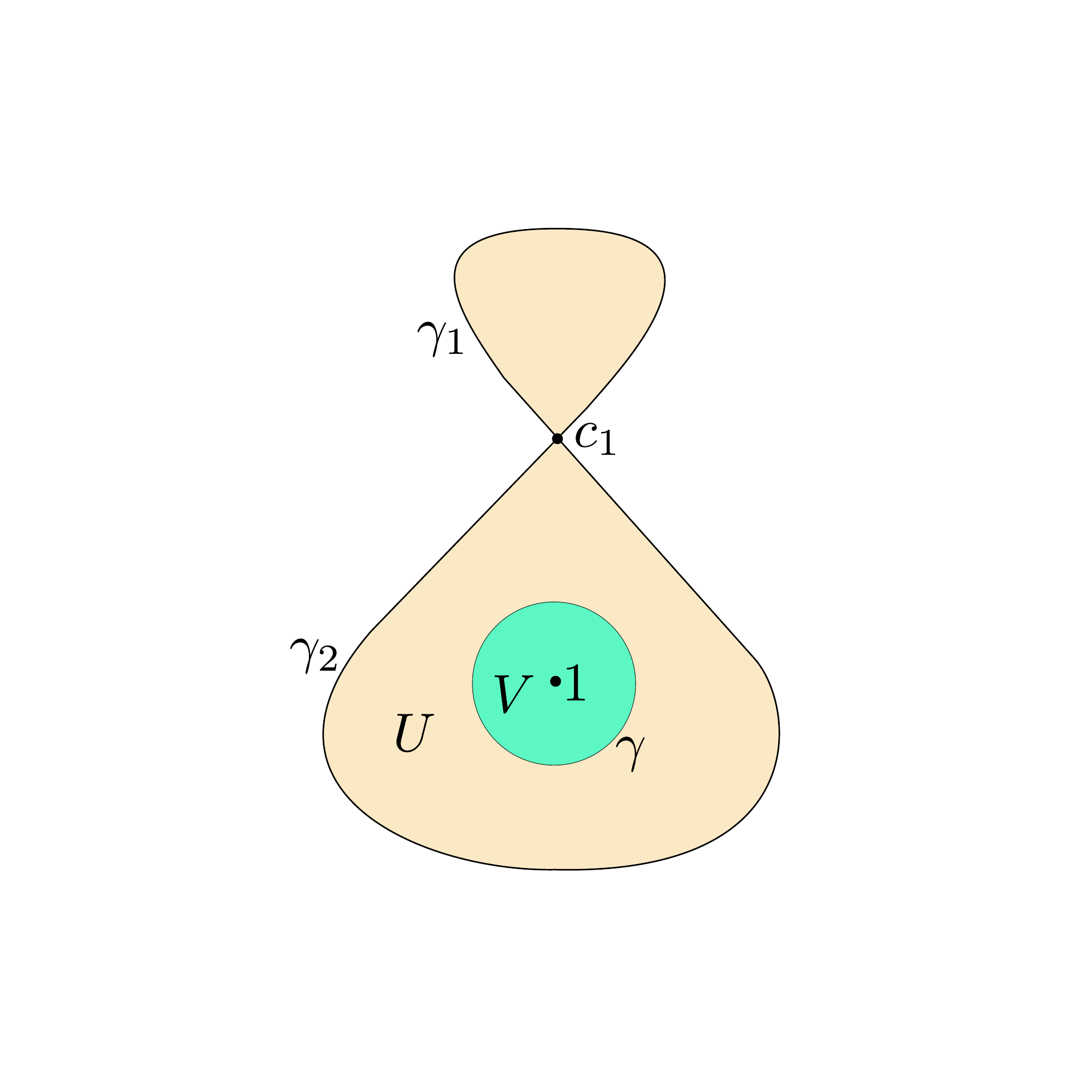}}
\hfill
\subfigure[$\partial U=\gamma_1 \cup \gamma_2$.]{\includegraphics[width=7cm]{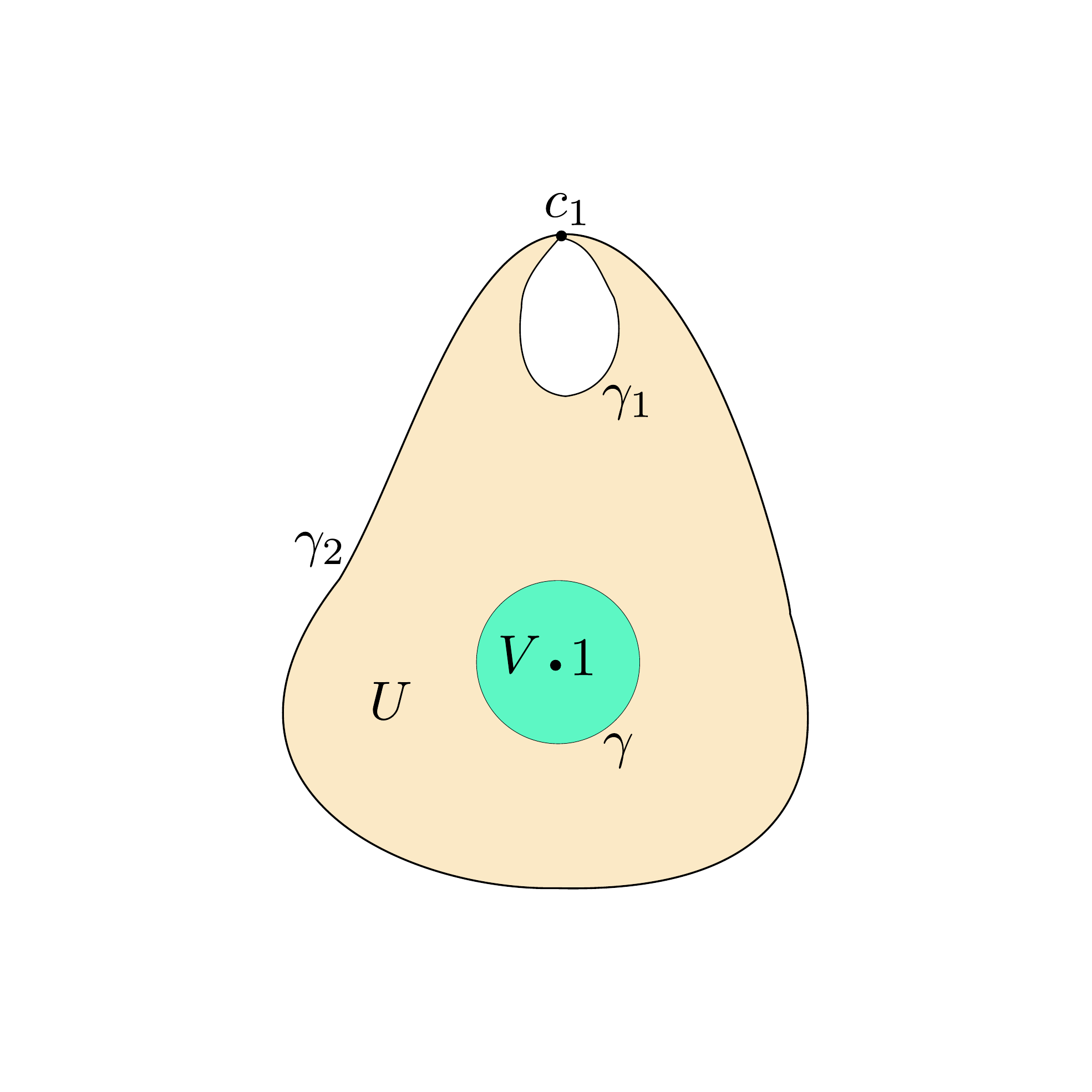}}
\hfill

\caption{The two possible configurations, of preimages of $\gamma$, described in \thref{propcurves}.}
\label{fig:posibilitati}
\end{figure}
\begin{figure}
\centering
\includegraphics[width=80mm]{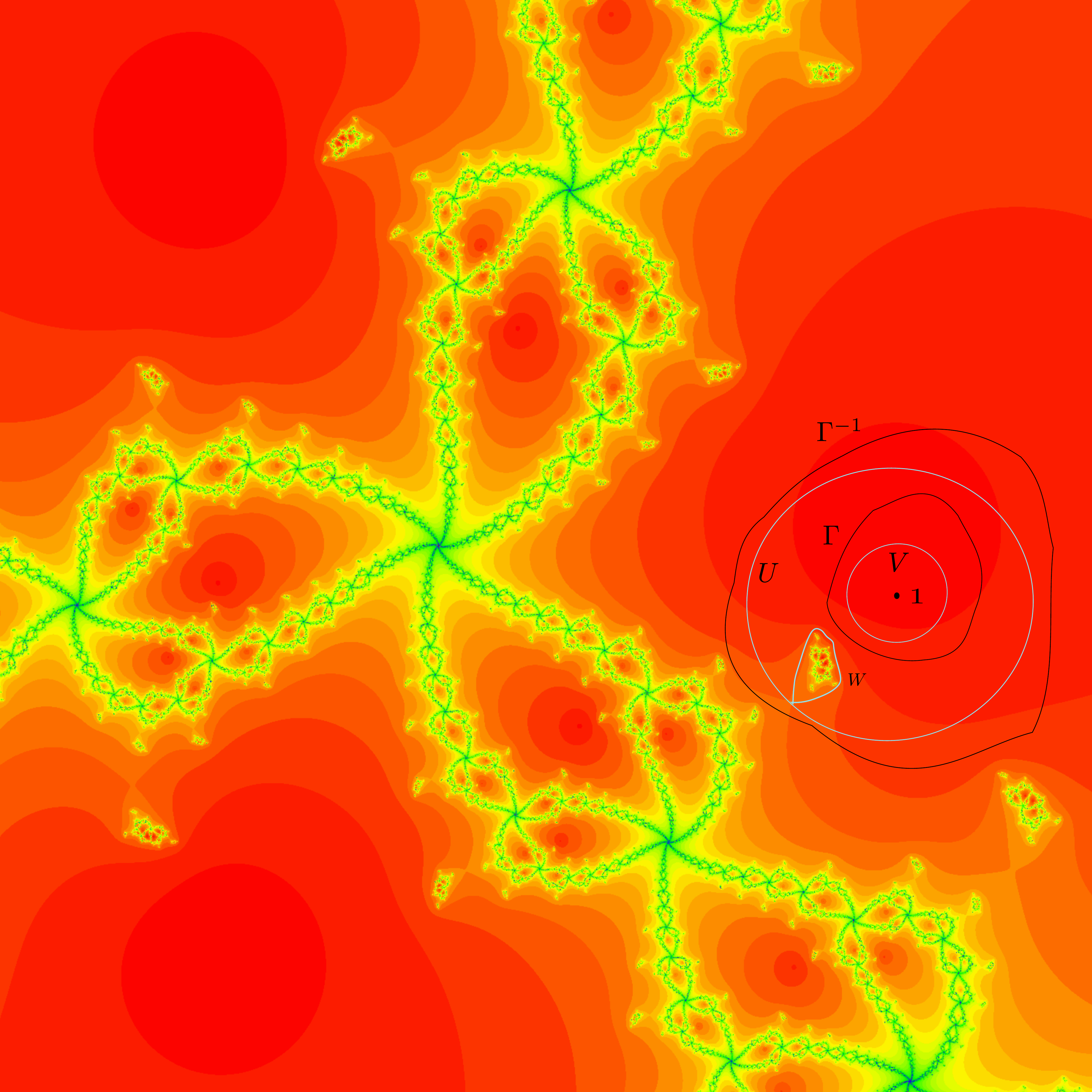}
\caption{Description of the situation in \thref{propcurves}, where $n=3$ and $\alpha=0.2+1.592i$.}
\label{fig:poc}
\end{figure}

\begin{proof}
Let $U$ be the maximal domain of definition of the B{\"o}ttcher coordinates of the superattracting fixed point $z=1$. By hypothesis and \thref{cortrich}, $\immo$ contains the critical point $c_1\coloneqq c_{n, \, a ,\, 1}$ \Big(see \eqref{exprcpjn} \Big), which lies on $\partial U$, and no other preimages of $z=1$.  Since $z=1$ has local degree $3$, the map $\jn|_{\immo} : \immo \to \immo$ has degree $3$. Let $V\coloneqq \jn (U)$. Then $\gamma\coloneqq \partial V$ is a Jordan curve.  Let $\gamma^{-1}$ be the preimage of $\gamma$ contained in $\immo$. Then $\gamma^{-1}= \gamma_{1} \cup \gamma_{2}$ is the union of two simple closed curves which intersect at the critical point $c_1$. Exactly one of the two curves, say $\gamma_{2} $, contains $z=1$ in the Jordan domain bounded by it. There exist two possibilities: either $\partial U=\gamma_{2}$, or $\partial U=\gamma_{1} \cup \gamma_{2}$ (see Figure \ref{fig:posibilitati}). Assume that $\partial U=\gamma_{2}$. By hypothesis, the critical point $c_1$ lies in $\immo$ and $\gamma_{1}$ is contained in the Fatou set. Therefore, $\gamma_{1}\subset \immo$ and there exists a preimage of $V$ which lies in $\textrm{Int}(\gamma_{1})$. Hence, $\immo$ contains a preimage of $z=1$ other than itself, which is impossible according to \thref{cortrich}. Consequently, we have that $\partial U=\gamma_{1} \cup \gamma_{2}$. Let $W$ be the connected component of $\rs\setminus \gamma_{1}$ which does not contain $z=1$. Then $W$ is mapped to an open set which contains $\rs \setminus U$, so $W$ contains a pole. Since $z=0$ is mapped to $z=\infty$ with degree $n-1$, $z=\infty$ is a fixed point, and the map has degree $2n$, there remain exactly $n$ preimages for $z=\infty$. By symmetry, the pole in $W$ is simple; therefore, $\partial W$ is mapped onto $\partial V$ with degree $1$. Hence, $\gamma_{1}$ is mapped onto $\gamma$ with degree $1$ and $\gamma_{2}$ is mapped onto $\gamma$ with degree $2$.

Let $\Gamma$ be an analytic Jordan curve which surrounds $z=1$ such that $\Gamma \subset U\setminus V$, and let ${\mathcal{A}}$ be the open annulus bounded by $\Gamma$ and $\gamma$. Then $\mathcal{A}$ has precisely $3$ preimages in $\immo$. Since $\mathcal{A}$ does not contain any critical value, its preimages do not contain critical points. It follows from the Riemann-Hurwitz formula (see, for instance, \cite{Ste}) that any preimage of $\mathcal{A}$ is also an annulus. One preimage of $\mathcal{A}$ lies in $W$ and is mapped onto $\mathcal{A}$ with degree $1$. There exists precisely one other preimage of $\mathcal{A}$ in $\immo$, which we denote by $\mathcal{A}^{-1}$. It lies in $\immo \setminus \textrm{Fill}(U)$, surrounds $z=1$, and is mapped onto $\mathcal{A}$ with degree $2$. Let $\Gamma^{-1}$ be the boundary component of $\mathcal{A}^{-1}$ which is mapped onto $\Gamma$. Observe that $\Gamma^{-1}$ is an analytic Jordan curve. Since $\Gamma^{-1}$ surrounds $z=1$ and lies outside $U$, we have that $\Gamma \subset \textrm{Int}(\Gamma^{-1})$  (see Figure \ref{fig:poc}).
\end{proof}

The main tool used in the proof of Theorem A is quasiconformal surgery. For an introduction of the topic, we refer to \cite{BF}. The strategy of the proof is as follows. We start by defining a quasiregular map $f:A_{n,\alpha}^{*}(1)\to A_{n,\alpha}^{*}(1)$ on the immediate basin of attraction of $z=1$, which we later extend to a quasiregular map $F:\rs \to \rs$. Secondly, we construct a {\it symmetric} $F$-invariant Beltrami coefficient and prove, using the Integrability Theorem (see, for instance, \cite{BF}, Theorem 1.28), the existence of a map $N_P$ quasiconformally conjugate to $F$. Then, we use \thref{tanlei} to show that $N_P$ is a map obtained by applying Newton's method to a polynomial of degree $n$, and it is quasiconformally conjugated to $N_{f_{n}}$. Finally, we compare the filled Julia sets of $N_{f_{n}}$ and $\jn$.

\begin{proof}[Proof of Theorem A]
Let $0<\rho <1.$ Let ${\it R}:\textrm{Int}(\Gamma)\to \mathbb{D}_{\rho^2}$ be a Riemann map such that ${\it R}(1)=0$. Since $\Gamma$ is an analytic curve, the Riemann map ${\it R}$ extends analitically to the boundary (see, for instance, \cite{BF}, Theorem 2.9c). Let $\psi_2:\Gamma\to \mathbb{S}_{\rho^2}$ be the restriction of ${\it R}$ to its boundary. Let $\psi_1:\Gamma^{-1} \to \mathbb{S}_{\rho}$ be an analytic lift map such that $\psi_2 \big( \jn (z)\big) = \big(\psi_1(z)\big)^2$. Let $\mathcal{A}=\textrm{Int}(\Gamma^{-1})\setminus \overbar{\textrm{Int}(\Gamma)}$ and $\mathcal{A}_{\rho^2, \rho}=\mathbb{D}_{\rho}\setminus \overbar{\mathbb{D}_{\rho^2}}$. Let $\psi: \partial  \mathcal{A} \to \partial \mathcal{A}_{\rho^2,\,  \rho}$, such that $\psi|_{\Gamma^{-1}}=\psi_1$ and $\psi |_{\Gamma}=\psi_2$. Since $\psi_1$ and $\psi_2$ are analytic maps, $\psi$ extends quasiconformally to $\psi:\overbar{\mathcal{A}} \to \overbar{\mathcal{A}}_{\rho^2, \, \rho}$ (see, for instance, \cite{BF}, Proposition 2.30b).

We now define $f:\immo \to \immo$ quasiregular, as follows:

\[ 
f(z)\coloneqq \left\{
\begin{array}{ll}
      {\it R}^{-1} \big(({\it R}(z))^2\big) & \textrm{if }z \in \textrm{Int}(\Gamma), \\
      {\it R}^{-1} \Big( \big(\psi(z)\big)^2\Big) & \textrm{if }z\in  \overbar{\mathcal{A}},\\
      \jn(z) & \textrm{if }z \in A_{n,\alpha}^{*}(1)\setminus \overbar{\textrm{Int}(\Gamma^{-1})}. \\
\end{array} 
\right. 
\]
Now let $\xi\coloneqq e^{\frac{2 \pi i}{n}}$. We have that $I_{\xi^{j}}(z)=I_{\xi}^{j}(z)$, for $j \in \{0, \, 1, \, \dots\, \, n-1\}$, where $I_{\xi}$ is defined as in \thref{rotationwithn}. We extend to a quasiregular map $F:\rs \to \rs$ defined on the Riemann sphere, as follows:
\[
F(z)\coloneqq \left\{
\begin{array}{ll}
      I_{{\xi}^{j}}\circ f \circ I_{{\xi}^{j}}^{-1}(z) & z \in A_{n,\alpha}^{*}(\xi^j), \, j\in \{ 0, \, 1, \, ..., n-1\}, \\
      \jn(z) & \textrm{otherwise}. \\
\end{array}  
\right.
\]

Observe that $F$ is a quasiregular map which coincides with $\jn$ outside the immediate basins of the roots of the unity. We intend to construct an $F-$invariant and $I_{\xi}-$ invariant Beltrami coefficient $\mu$. We first define an $F$-invariant Beltrami coefficient, say $\mu_1$, in $\mmo$, as follows:
 \[
\mu_1(z)\coloneqq \left\{
\begin{array}{ll}
      \psi^{*}\mu_0(z)  & \textrm{if }z \in \mathcal{A},\\
      (F^m)^{*}\mu_0 (z)  & \textrm{if } F^{m-1} (z)\in \mathcal{A},\textrm{ for }m\ge 2,\\
      \mu_0(z) & \textrm{otherwise}.
\end{array} 
\right. 
\]
Observe that for $z \in \mathcal{A}$, we have that $\phi^* \mu _0 (z) =F^{*} \mu_0(z)$. Now, we extend the previous construction to the rest of the Fatou set, that is, the basins of attraction of the $n$-th root of the unity $\xi^j\ne 1$, for $1\le j \le n-1$. We define an $I_{\xi}-$invariant Beltrami coefficient in $A_{n,  \alpha}(\xi^{j})$:
 \[
\mu(z)\coloneqq \left\{
\begin{array}{ll}
    \mu_1(z)  & \textrm{if }z \in A_{n,  \alpha}(1),\\
    (I_{\xi^j}^{-1})^{*} \mu_1(z)  & \textrm{if }z \in A_{n, \alpha}(\xi),\\
    \mu_0(z) & \textrm{otherwise}.
\end{array} 
\right. 
\]

For $z \in A_{n,  \alpha}(\xi)$ we have that $$(F)^{*} \mu= (F)^{*} (I_{\xi^j}^{-1})^{*}\mu_1= (I_{\xi^j}^{-1}F)^{*}\mu_1=(F I_{\xi^j}^{-1})^{*} \mu_1=(I_{\xi^j}^{-1})^{*} F^{*}\mu_1= (I_{\xi^j}^{-1})^{*} \mu_1=\mu.$$
It follows that $\mu$ is also $F-$invariant. By hypothesis, the map $\jn$ is hyperbolic, hence, the Julia set has measure $0$. Since $I_{\xi}^n(z)=z$, by construction, $\mu$ is both $F-$invariant and $I_{\xi}^{-1}-$invariant, with bounded dilatation. By the Integrability Theorem (see, for instance, \cite{BF}, Theorem 1.28), there exists $\phi_0:\rs\to \rs$ quasiconformal map such that $\phi_{0}^{*} \mu_0=\mu$. We normalize $\phi_0$ such that $\phi_0(0)=0$, $\phi_0(\infty)=\infty$, and $\phi_0$ is tangent to the identity at $\infty$. Let $\phi_{\xi}:=I_{\xi}\phi_0 I_{\xi}^{-1}$. We prove that $\phi_{\xi}$ and $\phi_0$ coincide by using the uniqueness part of the Integrability Theorem. First, we have that $\phi_{\xi}$ satisfies the same equation as $\phi_0$:
$$ \phi_{\xi}^{*}\mu_0=(I_{\xi}^{-1})^{*}\phi_0^{*}I_{\xi}^{*}\mu_0=(I_{\xi}^{-1})^{*}\phi_0^{*}\mu_0=(I_{\xi}^{-1})^{*}\mu=\mu.$$

We have that $\phi_{\xi}$ satisfies $\phi_{\xi}^{*} \mu_0 =\mu$, $\phi_{\xi}(\infty)=\infty$, $\phi_{\xi}(0)=0$, and $\phi_{\xi}$ is tangent to the identity at $\infty$. It follows from the uniqueness up to post-composition with M{\"o}bius transformations of the Integrability Theorem that $\phi_{\xi}=\phi_0$; therefore, $I_{\xi}\circ \phi_{0}=\phi_{0} \circ I_{\xi}$. 

Now let $N_P:\rs \to \rs$, $N_P:=\phi_{0}\circ F \circ \phi_{0}^{-1}$. Observe that, by construction, $N_P\circ I_{\xi}=I_{\xi} \circ N_P$. The map $N_P$ is quasiregular and satisfies $(N_P)^{*}\mu_0=\mu_0$, therefore, by Weyl's lemma (see, for instance, \cite{BF}, Theorem 1.14) it is holomorphic and quasiconformally conjugated by $\phi_0$ to $F$. Since $z= \infty$ is a fixed point of $F$ which is topologically repelling, $z= \infty$ is a repelling (therefore, not superattracting) fixed point of  $N_P$. It also follows from the conjugacy that $N_P$ has precisely $n$ distinct superattracting fixed points, the set $\{\xi^{j}\phi_0(1)\}$, where $j \in \{ 0,\, 1, \dots\, n-1\}$.

By \thref{tanlei}, the map $N_P$ is the map obtained by applying Newton's method to $$P(z)=\displaystyle{\prod_{j=0}^{n-1}\Big(z-\xi^j \phi_{0}(1)\Big)=z^n-\phi_0^n(1)}.$$Since $P$ and $f_n$ are linearly conjugated by $\eta(z)=\phi_{0}(1)z$, it can be proven (similarly to \thref{jordica}) that $N_P$ and $N_{f_n}$ are also linearly conjugated. Finally, the maps $F$ and $N_{f_{n}}$ are quasiconformally conjugated.

The Julia set of $N_{f_{n}}$, $J (N_{f_{n}})$, is connected (see \cite{Shi0}, Theorem 3.1). Moreover, by construction, $N_{f_{n}}$ and $\jn$ are quasiconformally conjugate in a neighborhood of $J (N_{f_{n}})$, by a conjugacy, say $\phi$. Since the conjugacy sends $\infty$ to $\infty$, we can conclude that there is an unbounded connected component $\Pi$ of $J (\jn)$, which is a quasiconformal copy of $J (N_{f_{n}})$. The fact that $\phi(J (N_{f_{n}}))$ is a connected component of $J ( \jn)$ follows from the surgery construction, since the surgery is done on the Fatou set of $\jn$.
\end{proof}

\section{Proof of Theorem B}
\label{sec:ThB}
We begin by studying the case of $n=2$. Let $\alpha>2$ and let $M_2(z) \coloneqq \frac{z+1}{z-1}$ be the M{\"o}bius transformation which maps the superattracting fixed points $z=1$ and $z=-1$, to $z=\infty$ and $z=0$. Finally, set $a=2(\alpha-1)>2$, and consider the map 
\begin{equation}
B_{a}(z)\coloneqq (M_2\circ\jndoi\circ M_{2}^{-1})=z^3\frac{z-a}{1-az}.
\end{equation}
The map $\jdoi(z)=z^3\frac{z-a}{1-az}$ is a rational map of degree $4$ studied in \cite{CFG1}, \cite{CFG2}, and \cite{CCV}. In \cite[Section 4]{CCV} it is proven that for $a\in \mathbb{C}$, $|a|>15.133$, $c_{+}\in \mmodoi$. More precisely, it is shown that $B_{a}(c_{+}) \in \immodoi$. We will prove that this is a sufficient condition for $\immodoi$ to be infinitely connected. Therefore, to prove Theorem B when $n=2$, it suffices to prove the statements for the family $\jdoi$, and by conjugacy, they hold for $\jndoi$.

The map $B_a$ is a rational map of degree $4$, and it is symmetric with respect to $\mathbb{S}^1$. The points $z=0$ and $z=\infty$ are superattracting fixed points of local degree $3$. Moreover, $z_{\infty}=\frac{1}{a}\in (0, \, 1)$ is a pole, and $z_0=a$ is a preimage of $z=0$. Consequently, there are two free critical points given by
\begin{equation}
\label{exprcp}
c_{\pm}= \frac{1}{3 a}\Big( 2+a^2\pm \sqrt{(a^2-4)(a^2-1)} \Big).
\end{equation}

The following lemma is a particular case of \cite[Proposition 4.5]{CCV}.

\begin{lemma}
\thlabel{computation2}
Let $a>1$. If $|z|>2a$, then $z \in \immodoi$. Equivalently, for $a>1$, we have that $\rs\setminus \overbar{ \mathbb{D}}(0, \, 2a) \subset \immodoi$.

\end{lemma}
\begin{proof}
If $|z|>2a$ then $$ |\jdoi(z)|=|z^3|\frac{|z-a|}{|1-az|}>|z-a||z|\frac{2a|z|}{|1-az|}>a|z| \frac{2a|z|}{1+a|z|}>a|z|.$$
Since $|\jdoi(z)|>|z|$, it follows that $z\in \immodoi$. 

\end{proof}

In the proof of Proposition 4.6 in \cite{CCV}, the authors show that for $ a \in \mathbb{C}$ with $|a|$ large enough (indeed $|a|>16$), we have $\jdoi(c_{+})\in \immodoi$. A similar proof was previously done in \cite[Lemma 2.6]{CFG2} for a family that includes $B_a$ (but without providing an explicit bound). Here we present an easier proof, only for real values of the parameter $a$.

\begin{lemma}
\thlabel{computation3}
If $a \in \mathbb{R}_{+}$ is large enough, then $\jdoi(c_{+})\in \immodoi$.

\end{lemma}

\begin{proof}
It follows from \eqref{exprcp} that if $a>2$, then $ \frac{a}{2}<c_{+}<a$. Consequently, $1-ac_{+}<0$ and $|1-ac_{+}|<2a^2$. 

Also, we have that $$a^2<2a^2-2<2+a^2+ \sqrt{(a^2-4)(a^2-1)}<2a^2.$$ It follows from the right hand side of the inequality that 
$$ |2+a^2+ \sqrt{(a^2-4)(a^2-1)}-3a^2|>a^2.$$ Since $|\jdoi (c_{+})|=c_{+}^3 \frac{|c_{+}-a|}{|1-ac_{+}|} $, then
$$ |\jdoi(c_{+})|=\frac{1}{27 a^3}\Big( 2+a^2+ \sqrt{(a^2-4)(a^2-1)} \Big)^3 \frac{1}{ 3a}\frac{|2+a^2+ \sqrt{(a^2-4)(a^2-1)}-3a^2|} {|1-ac_{+}|}.$$
Hence, for $a>2$ we have
$$ |\jdoi(c_{+})|>\frac{1}{27 a^3} a^6  \frac{1}{3a} \frac{a^2}{2a^2}=\frac{a^2}{162}.$$
Clearly, for $a$ large enough ($a>400$), we have $\frac{a^2}{162}>2a$. According to \thref{computation2}, we conclude that $\jdoi(c_{+})\in \immodoi$.

\end{proof}
\begin{prop}
\thlabel{infconnn=2}
Assume $a\in \mathbb{R}_{+}$ is large enough such that \thref{computation3} applies. Then $c_{+}\in \immodoi$ and $\immodoi$ is infinitely connected.
\end{prop}
\begin{proof}
Observe that, for $a>1$, we have that $0<z_{\infty}<z_0<2a$. By \thref{computation3}, $\jdoi(c_{+})\in \immodoi$. Therefore, the critical point $c_{+}$ lies either in $\immodoi$ or in a preimage of $\immodoi$.

Assume that $c_{+}$ lies in a preimage of $\immodoi$, distinct from $\immodoi$, say $U$. Since $U$ contains a critical point, it is mapped onto $\immodoi$ with degree at least $2$. Hence, $U$ contains at least $2$ preimages of $z=\infty$ (different from itself), a contradiction with $\deg(\jdoi)=4$, and $z= \infty$ being a superattracting fixed point with local degree $3$. 

Finally, let $f:[0, \, 1] \to  \rs$ be an injective map, analytic on $(0, \, 1)$ such that $f(0)=c_{+}$, $f(1)=\infty$, $\textrm{Im}(f(x))>0$, for $x\in (0, \, 1)$, and $f(x)\in \immodoi$. Let $\gamma \subset \immodoi$ be the graph of the function $f$. It follows from the Schwarz reflection principle that there exists $\gamma_1\subset \immodoi$ symmetric with respect to the real line to $\gamma$. Let $V$ be the subset of the Riemann sphere bounded by $\gamma$ and $\gamma_1$ and containing the interval $[c_{+}, \, \infty]$. Recall that $z_0 \in A_a(0)$. Since $a\in V$, and $\partial V \subset \immodoi$, we have that $\immodoi$ is multiply connected. Therefore, $\immodoi$ is infinitely connected.

\end{proof}

\begin{remark}
\thlabel{rem2}
It follows from \cite[Proposition 4.6]{CCV} that all $a \in \mathbb{C}$, with $|a|>15.133$, belong to the same hyperbolic component. Since the connectivity of $\immodoi$ is an invariant topological property within hyperbolic components, we conclude from \thref{infconnn=2} that if $|a|>15.133$, then $\immodoi$ is infinitely connected. This completes the proof of Theorem B for $n=2$.
\end{remark}

To finish the proof of Theorem B we now consider $n\ge 3$. As we did before, we consider a new map $\jnconju$ which is conjugated to $\jn$ via the M{\"o}bius map $M(z)= \frac{1}{z-1}$. More specifically, we consider $\jnconju:\rs\to \rs$, given by $\jnconju=M\circ \jn \circ M^{-1}.$ Since $M$ sends $z=1$ to $z=\infty$ and $z=\infty$ to $z=0$, it is clear from \eqref{exprjn} that $z=\infty$ is a superattracting fixed point of $\jnconju$ with local degree $3$ and $z=0$ is a fixed point of $\jnconju$. Doing some computations, the expression of $\jnconju$ writes as follows
\begin{equation}
\label{exprjnconju}
\jnconju (z)= \frac{2nz(z+1)^{n-1}[E_1(\alpha) z^{n}+E_2(\alpha)(z+1)^n]}{Q_{2n-3}^1(z)+\alpha Q_{2n-3}^2(z)}\eqqcolon\frac{zP(z)}{Q(z)},
\end{equation}
where $E_1(\alpha)\coloneqq -\alpha(n-1)$, $E_2(\alpha)\coloneqq \alpha (n-1)-n $, and $Q_{2n-3}^j (z) $, $j=1,\, 2$ are degree $2n-3$ polynomials, with coefficients independent of $\alpha$. We split the proof of the case $n \ge 3$ in several lemmas. We start by giving an estimate for a zero of $\jnconj$, which lies on the positive real line.
\begin{lemma}
\thlabel{zeron}
Let $\alpha>2$ and let $S(z)=E_1(\alpha)z^{n}+E_2(\alpha)(z+1)^n$ .Then $S\Big( \alpha(n-1) \Big)<0<S(\alpha n -\alpha -n)$. In particular, $ \jnconj$ has a zero on the interval $\Big(\alpha(n-1)-n, \, \alpha(n-1)\Big)$, for all $\alpha>2$.
\end{lemma}
\begin{proof}
Direct computations show that $S$ writes as
\begin{equation}
\label{exprS}
S(z)=-nz^n+[\alpha (n -1)-n) ]\sum\limits_{k=0}^{n-1} {n \choose k} z^k. 
\end{equation}
On the one hand, $$S\Big(\alpha (n -1) -n\Big)= \sum\limits_{k=0}^{n-2} {n \choose k}(\alpha n -\alpha -n)^k>0.$$ 
On the other hand, $$S\Big(\alpha(n-1)\Big)=-n\Big[\alpha (n-1) \Big]^n+[\alpha (n -1) -n]\sum\limits_{k=0}^{n-1} {n \choose k}[\alpha(n-1) ]^{k}$$
$$=-n\Big[\alpha (n-1) \Big]^n+[\alpha (n -1) ]\sum\limits_{k=0}^{n-2} {n \choose k}[\alpha(n-1) ]^{k}+[\alpha(n-1)] n\Big[\alpha (n-1) \Big]^{n-1}-n\sum\limits_{k=0}^{n-1} {n \choose k}[\alpha(n-1) ]^{k} $$
$$=\sum\limits_{k=0}^{n-2} \Big[{n \choose k}-n{n \choose k+1}\Big][\alpha(n-1) ]^{k+1}-n{n \choose 0}<0.$$
\end{proof}

The following technical lemma will be useful later.
\begin{lemma}
\thlabel{lol}
Let $m, k\in \mathbb{N}^{*}$, $m>k$. Let $u, v_{j}\in \mathbb{C}$, $j=1, \dots, m$. If $|u|-\sum \limits_{j=1}^{m}|v_{j}|>0$, then $$\left|u-\sum \limits_{j=1}^{k}v_{j}\right|>\left|\sum \limits_{j=k+1}^{m}v_{j}\right|.$$
\end{lemma}
\begin{proof}
Since $|u|-\sum \limits_{j=1}^{m}|v_{j}|>0$, we have that $$\left|u-\sum \limits_{j=1}^{k}v_{j}\right|\ge|u|-\sum \limits_{j=1}^{k}|v_{j}|>\sum \limits_{j=k+1}^{m}|v_{j}|\ge \left|\sum\limits_{j=k+1}^{m}v_{j}\right| .$$ 
\end{proof}

We give a sufficient condition for points to lie in $\immR$.

\begin{lemma}
Let $\alpha>0$ large enough. If $|z|>n\alpha$, then $z \in \immR$.
\thlabel{nmodul}
\end{lemma}

\begin{proof}

We show that if $|z|>n\alpha$, then $|\jnconj (z)|>|z|$, which is a sufficient condition for $z \in \immR$. According to \eqref{exprjnconju}, we have to prove that, for $\alpha$ large enough, $\left| \frac{ P(z)}{Q(z)}\right|>1$. We write $P$ as
$$P(z)= 2n\Big[-nz^{2n-1}+n\alpha (n-1) z^{2n-2}+P_{2n-2}(z)+\alpha P_{2n-3}(z)\Big].$$
Observe that $P_{2n-2}(z)$ and $P_{2n-3}(z)$ are polynomials of degree $2n-2$ and $2n-3$, respectively, with coefficients independent of $\alpha$. For $\alpha$ large enough (recall that we are assuming $|z|>n\alpha$), the following statements hold:
\begin{enumerate}
\item{$(n-1)|z|^{2n-1}>n\alpha (n-1) |z|^{2n-2}$.}
\item{$\frac{1}{3}|z|^{2n-1}>|P_{2n-2}(z)|$, since \[\lim_{ \alpha \to \infty}\frac{n(n-1)|z|^{2n-2}}{|z|^{2n-1}}=0.\]}
\item{$\frac{1}{3}|z|^{2n-1} >|\alpha P_{2n-3}(z)|$, since \[\lim_{ \alpha \to \infty}\frac{P_{2n-3}(z)}{|z|^{2n-2}}=0.\]}
\item{$\frac{1}{3}|z|^{2n-1}>|Q(z)|$, since $Q_{2n-3}^1$ and $Q_{2n-3}^2$ are polynomials of degree $2n-3$ with coefficients independent of $\alpha$.}
\end{enumerate}
All together imply that
 $$n |z|^{2n-1}>n\alpha (n-1) |z|^{2n-2}+|P_{2n-2}(z)|+\alpha |P_{2n-3}(z)|+|Q(z)|.$$
By using \thref{lol} (remark that for $\alpha$ large enough, there is no root of $Q$ for $|z|>n\alpha$), we get that $$\left| \frac{ P(z)}{Q(z)}\right|=\left|2n \frac{-nz^{2n-1}+\alpha (n-1) z^{2n-2}+P_{2n-2}(z)+\alpha P_{2n-3}(z)}{Q(z)}\right|>2n>1.$$ 
Thus, for $|z|>n\alpha$, we have that $|\jnconj (z)|>|z|$ and $z\in \immR$.
\end{proof}

The following proposition concludes the proof of Theorem B.
\begin{prop}
\thlabel{finalprop}
Let $\alpha>0$ large enough. Then $\immR$ is infinitely connected.
\end{prop}

\begin{figure}

\centering
\includegraphics[width=80mm]{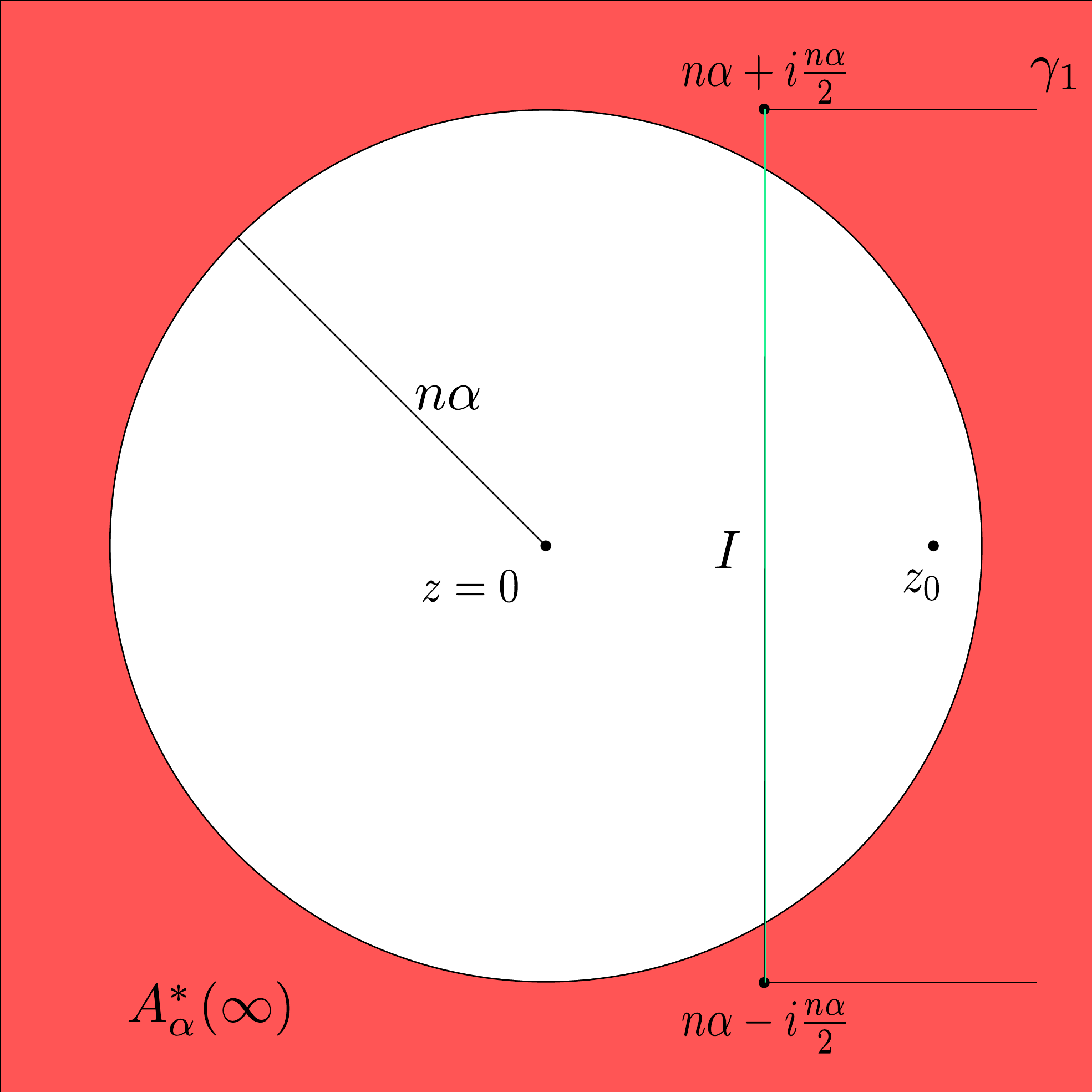}
\caption{Description of the situation in proof of \thref{finalprop}. The zero $z_0$ is separated by $(I \cup \gamma_1)\subset \immR$ from $z=0$. Therefore, $\immR$ is multiply connected.}
\label{fig:teoremaB}
\end{figure}
\begin{proof}
If $z \in(0, \, \alpha n-\alpha-n)$, then $n z^n<(\alpha n -\alpha -n)n z^{n-1}$. It follows that $S(z)$ (see \eqref{exprS}) has no zeros in $(0, \, \alpha n-\alpha-n)$. In particular, $\jnconj$ has no zeros in $(0, \, \alpha n-\alpha-n)$. Let $$I=\Big\{z \in \mathbb{C} |z=n\alpha\left(\frac{1}{2}+it\right) \Big\}, \, t \in [-1, \, 1].$$ We claim that $\jnconj (I)\subset \immR$.

Let $\jnconjt(z)\coloneqq \frac{1}{(1+z)^2}\jnconj (z) $. Firstly, we prove that there exists a constant $\kappa>0$ such that for $z\in I$, we have that $\Big|\jnconjt(z)\Big|>\kappa$. A direct computation shows that 
$$\left|\jnconjt\left(n\alpha\left(\frac{1}{2}+it\right) \right)\right|\coloneqq \frac{N(\alpha)}{M(\alpha)},$$
where $N$ and $M$ are polynomials of degree $2n-2$ in $\alpha$ with coefficients depending on $t$. Moreover, if we denote by $c(t)$ the degree $2n-2$ coefficient of $N$, we have:
$$c(t)=2n^{2n-1}\left(\frac{1}{2}+it\right)^{2n-2}\left[-n\left( \frac{1}{2}+it\right)+n-1\right].$$ Observe that $$\min_{t\in[-1, \, 1]} \Big|c(t)\Big|=\Big|c(0)\Big|\coloneqq C>0.$$ We denote by $d(t)$ the degree $2n-2$ coefficient of $M$. Let $D\coloneqq \max\limits_{t\in[-1, \, 1]}|d(t)|,$ and let $\kappa \coloneqq\frac{C}{2D}$.
For large enough $\alpha$, we have that $\Big|\jnconjt\Big(n\alpha\left(\frac{1}{2}+it\right) \Big)\Big|>\kappa$ and that $$\left|\jnconj\left(n\alpha\left(\frac{1}{2}+it\right) \right)\right|>\kappa\left|\frac{n\alpha}{2}+1+n\alpha t i\right|^2>\frac{n^2}{4}\kappa \alpha^2 >n\alpha.$$
It follows from \thref{nmodul} that $\jnconj (I)\subset \immR$. Hence, $I$ is a subset of $\immR$ or a preimage of this Fatou component. Moreover, for $z_{\pm}=\frac{n}{2}\alpha+in\alpha$ we have $|z_{\pm}|>n\alpha$. We conclude from \thref{nmodul} that $z_{\pm}\in \immR$. Therefore, $I \subset \immR$. By \thref{zeron}, we have that $\frac{n\alpha}{2}<z_0<n\alpha$. Therefore, there exists a piece-wise smooth curve $\Gamma= I\cup \gamma_1 \subset \immR$ such that $z_0 \in \textrm{Int}(\Gamma)$ and $0 \in \textrm{Ext}(\Gamma)$ (see Figure \ref{fig:teoremaB}). It follows that $\immR$ is multiply connected. By \thref{rem1}, it is infinitely connected. 
\end{proof}
\bibliography{bibliografia}
\bibliographystyle{amsalpha}
\end{document}